\documentclass{amsart}

\usepackage{amsmath}

\usepackage[pdftex]{graphicx}

\usepackage{amssymb}
\usepackage{ifthen}
\usepackage{enumitem}
\usepackage{grffile}

\usepackage[round]{natbib}

\newcommand{\R}{\mathbb R}
\newcommand{\E}{\mathbb E}

\newcommand{\N}{\mathbb N}

\renewcommand{\P}{\mathbb P}

\DeclareMathOperator{\Var}{Var}

\graphicspath{{./plots/}}

\theoremstyle{plain}

\newtheorem{satz}{Satz}[section]

\newtheorem{theo}[satz]{Theorem}

\theoremstyle{definition}

\theoremstyle{remark}
\newtheorem{remark}[satz]{Remark}

\newcommand{\dd}{{\rm d}}
\newcommand{\eee}{{\rm e}}

\begin{document}

\title[A characterization of the normal distribution]{A characterization of the normal distribution using stationary max-stable processes}

\author{Sebastian Engelke}
\address{Sebastian Engelke, Ecole Polytechnique F\'ed\'erale de Lausanne,
EPFL-FSB-MATHAA-STAT, Station 8, 1015 Lausanne, Switzerland
}
\email{sebastian.engelke@epfl.ch}
\author{Zakhar Kabluchko}
\address{Zakhar Kabluchko, Institute of Mathematical Statistics, University of M\"unster, Orl\'{e}ans-Ring 10, 48149 M\"unster, Germany}
\email{zakhar.kabluchko@uni-muenster.de}

\begin{abstract}
Consider the max-stable process
$\eta(t) = \max_{i\in\mathbb N} U_i \rm{e}^{\langle X_i, t\rangle - \kappa(t)}$, $t\in\mathbb{R}^d$,
where $\{U_i, i\in\mathbb{N}\}$ are points of the Poisson process with intensity $u^{-2}\rm{d} u$ on $(0,\infty)$, $X_i$, $i\in\mathbb{N}$, are independent copies of a random $d$-variate vector $X$ (that are independent of the Poisson process), and $\kappa:\mathbb{R}^d \to \mathbb{R}$ is a function. We show that the process $\eta$ is stationary if and only if $X$ has multivariate normal distribution and $\kappa(t)-\kappa(0)$ is the cumulant generating function of $X$. In this case, $\eta$ is a max-stable process introduced by R.\ L.\ Smith.

\end{abstract}

\maketitle

\section{Introduction}

The first class of max-stable processes that was used as a flexible
model for spatially distributed extreme events was introduced
by \cite{smi1990}. Let $N$ be a random vector having a zero-mean
$d$-variate normal distribution with covariance matrix $\Sigma$.
Further, let $N_i$, $i\in\N$, be independent copies of $N$.
Independently of the $N_i$'s, let $\{U_i, i\in\N\}$ be a Poisson point
process on $(0,\infty)$ with intensity $u^{-2}\dd u$.
The stochastic process
\begin{align}\label{smith}
M_\Sigma(t) := \max_{i\in\N}\, U_i \exp\left\{\langle N_i, t\rangle - \frac12 \langle t, \Sigma t\rangle  \right\}  , \quad t\in\R^d,
\end{align}
is now commonly termed the \textit{Smith process}. This process is
max-stable and stationary, where the former means that the pointwise
maximum $\frac 1n \max_{i=1}^n M_{\Sigma,i}$ of $n$ independent
copies $M_{\Sigma,1}, \dots, M_{\Sigma,n}$ of the process $M_\Sigma$
has the same finite-dimensional distributions as $M_\Sigma$ itself, for all $n\in\N$.
Recall also that a stochastic process $\{M(t), t\in\R^d\}$ is said to be stationary
if it has the same finite-dimensional distributions as the
shifted process $\{M(t+ h), t\in\R^d\}$, for all $h\in\R^d$.

\begin{figure}
  \begin{centering}
  \includegraphics[trim = 10mm 10mm 5mm 0mm, clip, width = 0.95 \textwidth]{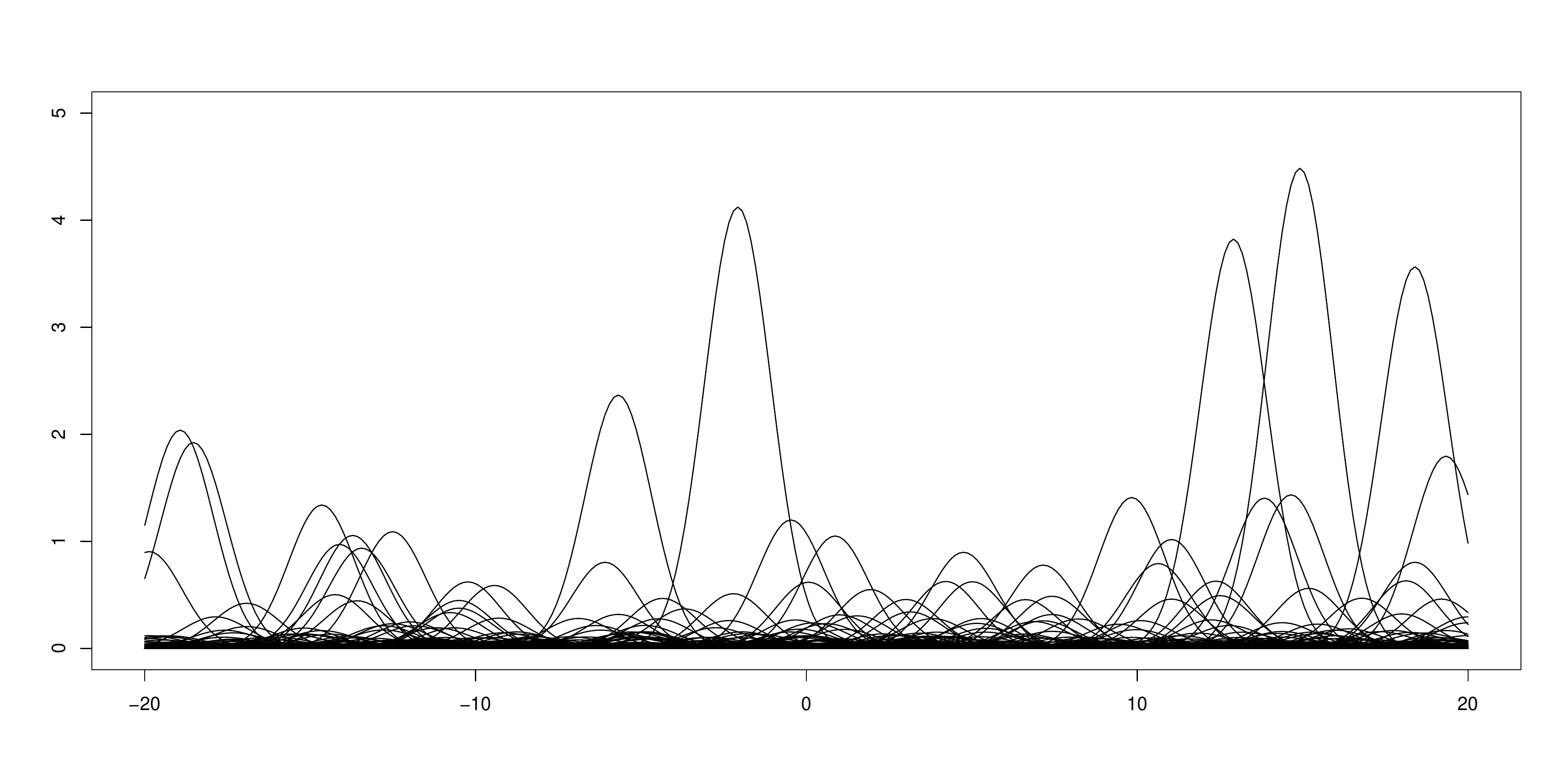}%
  \caption{Realization of the one-dimensional Smith process~\eqref{smith}.
  }
  \label{fig2}
  \end{centering}
\end{figure}

Thanks to its simple
form and the small number of parameters in low dimensions, the Smith process has become
a widely applied model in spatial extreme value statistics \citep{deh2006,eng2014,oes2015,pad2010,dav2012,wes2011} and it has been extended in several
directions \citep{deh2006,kab2009,smi2009,rob2013}. Smith processes also appeared in connection with convex hulls of independent and identically distributed samples~\citep{edg1981,hoh1996}. It is worth noting that
the class of Brown--Resnick processes introduced in \cite{kab2009}
includes the Smith process as a special case. In fact, instead
of the random parabolas $\langle N_i, t \rangle - \frac 12 \langle t , \Sigma t\rangle$
in \eqref{smith}, it is possible to consider independent copies
$Z_i$, $i\in\N$, of a zero-mean Gaussian process $\{Z(t), t\in\R^d\}$ with stationary increments and variance function $\sigma^2(t)=\Var Z(t)$.
The Brown--Resnick process \citep[cf.,][]{bro1977,kab2009}
\begin{align}\label{BR}
  M_{\text{BR}}(t) := \max_{i\in\N} \, U_i \exp\left\{ Z_i(t) - \frac 12 \sigma^2(t)   \right\}  , \quad t\in\R^d,
\end{align}
is stationary, max-stable and its distribution depends only on the so-called variogram $\gamma(t)= \Var (Z(t)-Z(0))$.
Choosing $Z$ to be the random linear Gaussian function $Z(t)=\langle N, t \rangle$ which has $\sigma^2(t) = \langle t, \Sigma t\rangle$, one recovers the Smith process.

In the case when the matrix $\Sigma$ is non-singular, the Smith process $M_\Sigma$ defined in~\eqref{smith}
can equivalently be represented as a moving maxima process \cite[cf.,][for instance]{wan2010}
\begin{align}\label{mmm}
  M_\Sigma(t) = \frac {\det(\Sigma)^{1/2}} {(2\pi)^{d/2}} \max_{i\in\N}\, V_i \exp\left\{ - \frac12 \langle (t-T_i), \Sigma(t- T_i)\rangle  \right\}.
\end{align}
Here, $\{(T_i, V_i), i\in\N\}$ is a Poisson point process with
intensity $\dd t \times v^{-2}\dd v$ on $\R^d\times (0,\infty)$.
In the moving maxima representation \eqref{mmm}, $M_\Sigma$ can be interpreted as the maximum over many ``storms'' where the $i$-th storm
has center point $T_i$ and strength $V_i$, and all storms have a common spatial shape
given by the $d$-variate normal
density with covariance matrix $\Sigma^{-1}$ (as opposed to $\Sigma$
in representation \eqref{smith}). Figure \ref{fig2} shows a one-dimensional stationary Smith process. Note in particular that because
of stationarity the origin does not play an exceptional role.

In this note we investigate the following question. If we drop the
assumption of Gaussianity in \eqref{smith} and replace the normally distributed random vector
$N$ by an arbitrary $d$-dimensional vector $X$, can we find a function
$\kappa: \R^d \to \R$ such that
the stochastic process
\begin{align}\label{eta}
  \eta(t) := \max_{i\in\N}\, U_i \exp\{\langle X_i, t\rangle - \kappa(t)\}  , \quad t\in\R^d,
\end{align}
is max-stable and stationary? Here, $\{U_i, i\in\N\}$ is a Poisson point
process on $(0,\infty)$ with intensity $u^{-2}\dd u$, and $X_i$, $i\in\N$, are independent copies of $X$ which are also independent of the Poisson process.
In fact, the max-stability of the process $\eta$
follows directly by the construction since~\eqref{eta} is the  de Haan representation~\citep{deh1984} of  the process $\eta$. The crucial part of the question above is the stationarity.

\section{Main result}
\begin{theo}\label{theo:main}
The process $\eta$ defined in~\eqref{eta} is stationary on $\R^d$ if and only if
$X$ has a $d$-variate normal distribution with some mean $\mu\in\R^d$ and covariance matrix $\Sigma$, and
\begin{equation}\label{eq:kappa}
\kappa(t)=\langle \mu, t\rangle +  \frac 12 \langle t, \Sigma t\rangle + \kappa(0).
\end{equation}
\end{theo}

\begin{proof}
If $X$ is $d$-variate normal with mean $\mu$ and covariance matrix $\Sigma$, and $\kappa(t)$ is given by~\eqref{eq:kappa}, then it is well known that $\eta$ is stationary; see for example Theorem~2 in \cite{kab2009}.

Let $\eta$ be stationary. We have to show that $X$ is normal and~\eqref{eq:kappa} holds.   Without restriction of generality let $\kappa(0) =0$ because otherwise we could replace $\kappa(t)$ by $\kappa(t) - \kappa(0)$ without changing the stationarity of $\eta$. Then, $\eta(0)$ has unit Fr\'echet distribution, i.e.,
$$
\P( \eta(0) \leq x) = \P(\max_{i\in\N} U_i \leq x) = \exp\{- 1/x\},
$$
for $x>0$. The stationarity of the one-dimensional margins of $\eta$ implies that
the cumulant generating function
\begin{align}\label{exp_moment}
    \varphi(t) := \log \E \eee^{\langle X, t\rangle}, \quad t\in\R^d,
\end{align}
is finite, and that $\kappa(t) = \varphi(t)$, for all $t\in\R^d$. Indeed,
  \begin{align*}
    \P( \eta(t) \leq x ) &=  \P\left( \max_{i\in\N}\, U_i \eee^{\langle X_i, t\rangle - \kappa(t)} \leq x \right)\\
    &= \exp\left\{- \int_0^\infty \P\left(\eee^{\langle X, t\rangle - \kappa(t)} > \frac xu  \right) \frac{\dd u}{u^2} \right\}\\
    &= \exp\left\{- \frac 1x \, \E \eee^{\langle X, t\rangle - \kappa(t)} \right\}.
  \end{align*}
  Thus, in fact, $\kappa(t) = \varphi(t)$ for all $t\in\R^d$.

  For $n\in\N$ and $t_1,\dots,t_n\in\R^d$, the cumulant generating function $\varphi_{t_1,\dots,t_n}$
  of the $n$-variate random vector $(\langle X, t_1\rangle - \varphi(t_1), \dots, \langle X, t_n\rangle  - \varphi(t_n))$
  is given by
  \begin{align}\label{laplace}
    \varphi_{t_1,\dots,t_n}(u_1,\dots, u_n)
    &:=
    \log \E \exp \left\{\sum_{i=1}^n (\langle X, t_i\rangle - \varphi(t_i))u_i\right\} \\
    &=
    \varphi\left(\sum_{i=1}^n u_it_i\right) - \sum_{i=1}^n u_i\varphi(t_i),
  \end{align}
  for all $u_1,\ldots, u_n\in\R$.
  Therefore, the general stationarity criterion for max-stable processes of the form~\eqref{eta} given in Proposition~6 in~\cite{kab2009}
  implies that $\eta$ is stationary, if and only if for all $h\in\R^d$ and all
  $u_1,\ldots, u_n\in[0,1]$ such that $\sum_{i=1}^n u_i = 1$ it holds that
  \begin{align}\label{laplace_cond}
    \varphi_{t_1,\dots,t_n}(u_1,\dots, u_n) = \varphi_{t_1+h,\dots,t_n+h}(u_1,\dots, u_n),
  \end{align}
  or, by \eqref{laplace},
  \begin{align}\label{stationarity}
     \varphi\left(\sum_{i=1}^n u_it_i\right) - \sum_{i=1}^n u_i\varphi(t_i)
     = \varphi\left(h + \sum_{i=1}^n u_it_i\right) - \sum_{i=1}^n u_i\varphi(t_i +h).
  \end{align}
  For arbitrary $t_1,t_2,h\in\R^d$ and $\delta\in[0,1]$ we have
  by \eqref{stationarity} that
  \begin{align*}
    \varphi\left( (1-\delta) t_1 + \delta t_2\right) - &(1-\delta)\varphi(t_1) - \delta\varphi(t_2)\\
     &= \varphi\left( (1-\delta)t_1 + \delta t_2 + h \right) - (1-\delta)\varphi(t_1+h) - \delta \varphi(t_2+h).
  \end{align*}
  Applying the above relation with $\varepsilon h$ instead of $h$ and rearranging the terms, we obtain that for every $\varepsilon > 0$,
  \begin{align*}
    &\frac{\varphi\left( (1-\delta)t_1 + \delta t_2  + \varepsilon h \right) - \varphi\left( (1-\delta) t_1 + \delta t_2\right)}{\varepsilon}\\
    &\hspace{3cm}= (1-\delta)\frac{\varphi(t_1+\varepsilon h) - \varphi(t_1)}{\varepsilon} + \delta \frac{\varphi(t_2+\varepsilon h) - \varphi(t_2)}{\varepsilon}.
  \end{align*}
  Note that the function $\varphi$ is infinitely differentiable by its definition~\eqref{exp_moment}. Sending $\epsilon \searrow 0$ on both sides gives
  \begin{align}\label{dot_equal}
    \langle\nabla \varphi\left( (1-\delta)t_1 + \delta t_2 \right), h\rangle =
    (1-\delta) \langle\nabla \varphi(t_1), h\rangle +  \delta\langle \nabla \varphi(t_2), h\rangle,
  \end{align}
  where $\nabla \varphi(t) \in \R^d$ denotes the gradient of $\varphi$
  at $t\in\R^d$.
  Since equation \eqref{dot_equal} holds for all $h\in\R^d$, we obtain
  \begin{align}\label{linear}
    \nabla \varphi\left( (1-\delta)t_1 + \delta t_2 \right) =
    (1-\delta) \nabla \varphi(t_1) +  \delta \nabla \varphi(t_2),
  \end{align}
  for all $t_1,t_2\in\R^d$ and $\delta\in[0,1]$. Thus, any of the $d$ components of the gradient $\nabla \varphi$ is both convex and concave and hence, an affine function.
  Consequently, there is a $d\times d$ matrix
  $\Sigma\in \text {Mat}_{d}(\R)$ and a vector $\mu\in\R^d$ such that
  \begin{align*}
    \nabla \varphi(t) = \Sigma  t + \mu.
  \end{align*}
  Observe that the function
  $\xi: \R^d \to \R$, $\xi(t) = \langle \mu, t\rangle + \frac 12\langle t, \Sigma t\rangle$, $t\in\R^d$, has the same gradient as $\varphi$. So, the difference between these two functions has gradient $0$ and we conclude that this difference is constant and, in fact, $0$ by the assumption $\varphi(0) = \kappa(0) = 0$. We obtain that
  \begin{align*}
    \kappa(t) = \varphi(t)  = \log \E\exp\{\langle X, t\rangle\} =  \langle \mu, t\rangle
    + \frac 12\langle t, \Sigma t\rangle \quad t\in\R^d.
  \end{align*}
It remains to observe that the matrix $\Sigma$ must be positive semidefinite because the function $\varphi$, being a cumulant generating function, is convex. Thus, $X$ has a $d$-variate normal distribution with mean vector $\mu$ and covariance
matrix $\Sigma$, which completes the proof.
\end{proof}

\begin{remark}\label{rem1}
For any random vector $X$ it is possible to make the one-dimensional distributions of the process $\eta$ defined in~\eqref{eta} stationary by choosing $\kappa(t)$ to be the cumulant generating function of $X$. However, the above proof shows that except in the case when $X$ is normal, it is not possible to make the two-dimensional distributions of $\eta$ stationary.
\end{remark}

\begin{remark}\label{rem2}
Similarly to the setting of the present paper one may ask whether in the definition of the Brown--Resnick process~\eqref{BR} we can replace $Z$ by some more general process (for example, a Gaussian process with non-stationary increments). 
This question was studied in~\cite{kab2010a}. 
\end{remark}

\begin{remark}
Theorem~\ref{theo:main} provides a characterization of the multivariate normal distribution based on max-stable processes.  See~\cite[page~151]{kotz} for a review of known characterizations of the multivariate normal law.
\end{remark}

\section*{Acknowledgements}
  Financial support from the Swiss National Science Foundation Project 161297 (first author) is gratefully acknowledged.

\bibliography{Engelke}

\begin{thebibliography}{17}
\providecommand{\natexlab}[1]{#1}
\providecommand{\url}[1]{\texttt{#1}}
\expandafter\ifx\csname urlstyle\endcsname\relax
  \providecommand{\doi}[1]{doi: #1}\else
  \providecommand{\doi}{doi: \begingroup \urlstyle{rm}\Url}\fi

\bibitem[Brown and Resnick(1977)]{bro1977}
B.~M. Brown and S.~I. Resnick.
\newblock Extreme values of independent stochastic processes.
\newblock \emph{J. Appl. Probab.}, 14:\penalty0 732--739, 1977.

\bibitem[Davison and Gholamrezaee(2012)]{dav2012}
A.~C. Davison and M.~M. Gholamrezaee.
\newblock Geostatistics of extremes.
\newblock \emph{Proc. R. Soc. Lond. Ser. A Math. Phys. Eng. Sci.},
  468:\penalty0 581--608, 2012.

\bibitem[de~Haan(1984)]{deh1984}
L.~de~Haan.
\newblock A spectral representation for max-stable processes.
\newblock \emph{Ann. Probab.}, 12:\penalty0 1194--1204, 1984.

\bibitem[de~Haan and Pereira(2006)]{deh2006}
L.~de~Haan and T.~T. Pereira.
\newblock Spatial extremes: models for the stationary case.
\newblock \emph{Ann. Statist.}, 34:\penalty0 146--168, 2006.

\bibitem[Eddy and Gale(1981)]{edg1981}
W.~F. Eddy and J.~D. Gale.
\newblock The convex hull of a spherically symmetric sample.
\newblock \emph{Adv. Appl. Probab.}, 13:\penalty0 751--763, 1981.

\bibitem[Engelke et~al.(2015)Engelke, Malinowski, Kabluchko, and
  Schlather]{eng2014}
S.~Engelke, A.~Malinowski, Z.~Kabluchko, and M.~Schlather.
\newblock Estimation of {H}\"usler--{R}eiss distributions and
  {B}rown--{R}esnick processes.
\newblock \emph{J. R. Stat. Soc. Ser. B Stat. Methodol.}, 77:\penalty0
  239--265, 2015.

\bibitem[Hooghiemstra and H\"usler(1996)]{hoh1996}
G.~Hooghiemstra and J.~H\"usler.
\newblock A note on maxima of bivariate random vectors.
\newblock \emph{Stat. Probab. Lett.}, 31\penalty0 (1):\penalty0 1--6, 1996.

\bibitem[Kabluchko(2010)]{kab2010a}
Z.~Kabluchko.
\newblock Stationary systems of {G}aussian processes.
\newblock \emph{Ann. Appl. Probab.}, 20:\penalty0 2295--2317, 2010.

\bibitem[Kabluchko et~al.(2009)Kabluchko, Schlather, and de~Haan]{kab2009}
Z.~Kabluchko, M.~Schlather, and L.~de~Haan.
\newblock Stationary max-stable fields associated to negative definite
  functions.
\newblock \emph{Ann. Probab.}, 37:\penalty0 2042--2065, 2009.

\bibitem[Kotz et~al.(2000)Kotz, Balakrishnan, and Johnson]{kotz}
S.~Kotz, N.~Balakrishnan, and N.~L. Johnson.
\newblock \emph{Continuous multivariate distributions. Vol. 1: Models and
  applications.}
\newblock New York, Wiley, 2nd edition, 2000.

\bibitem[Oesting(2015)]{oes2015}
M.~Oesting.
\newblock On the distribution of a max-stable process conditional on max-linear
  functionals.
\newblock \emph{Stat. Probab. Lett.}, 100:\penalty0 158--163, 2015.

\bibitem[Padoan et~al.(2010)Padoan, Ribatet, and Sisson]{pad2010}
S.~A. Padoan, M.~Ribatet, and S.~A. Sisson.
\newblock Likelihood-based inference for max-stable processes.
\newblock \emph{J. Amer. Statist. Assoc.}, 105:\penalty0 263--277, 2010.

\bibitem[Robert(2013)]{rob2013}
C.~Y. Robert.
\newblock Some new classes of stationary max-stable random fields.
\newblock \emph{Statist. Probab. Lett.}, 83:\penalty0 1496--1503, 2013.

\bibitem[Smith and Stephenson(2009)]{smi2009}
E.~L. Smith and A.~G. Stephenson.
\newblock An extended {G}aussian max-stable process model for spatial extremes.
\newblock \emph{J. Statist. Plann. Inference}, 139\penalty0 (4):\penalty0
  1266--1275, 2009.

\bibitem[Smith(1990)]{smi1990}
R.~L. Smith.
\newblock Max-stable processes and spatial extremes.
\newblock Unpublished manuscript, 1990.

\bibitem[Wang and Stoev(2010)]{wan2010}
Y.~Wang and S.~A. Stoev.
\newblock On the structure and representations of max-stable processes.
\newblock \emph{Adv. Appl. Probab.}, 42:\penalty0 855--877, 2010.

\bibitem[Westra and Sisson(2011)]{wes2011}
S.~Westra and S.~A. Sisson.
\newblock Detection of non-stationarity in precipitation extremes using a
  max-stable process model.
\newblock \emph{Journal of Hydrology}, 406:\penalty0 119 -- 128, 2011.

\end{thebibliography}
\bibliographystyle{plainnat}

\end{document}